\documentclass[10pt]{article}
\font\smallit=cmti10
\font\smalltt=cmtt10

\usepackage{caption}
\usepackage{color}
\usepackage{amssymb,latexsym,amsmath,epsfig,amsthm} 
\usepackage{empheq}
\usepackage{here}
\usepackage{url}
\usepackage{ascmac}
\makeatletter
\usepackage{here}
\usepackage{comment}

\renewcommand\section{\@startsection {section}{1}{\z@}
{-30pt \@plus -1ex \@minus -.2ex}
{2.3ex \@plus.2ex}
{\normalfont\normalsize\bfseries\boldmath}}

\renewcommand\subsection{\@startsection{subsection}{2}{\z@}
{-3.25ex\@plus -1ex \@minus -.2ex}
{1.5ex \@plus .2ex}
{\normalfont\normalsize\bfseries\boldmath}}

\renewcommand{\@seccntformat}[1]{\csname the#1\endcsname. }

\makeatother
\newtheorem{theorem}{Theorem}
\newtheorem{lemma}{Lemma}

\theoremstyle{definition}
\newtheorem{defn}{Definition}[section]

\newtheorem{remark}[theorem]{Remark}

\begin{document}

\begin{center}
\uppercase{\bf Impartial and Partizan Restricted Chocolate Bar Games}
\vskip 20pt
{\bf Ryohei Miyadera }\\
{\smallit Keimei Gakuin Junior and High School, Kobe City, Japan}\\
{\tt runnerskg@gmail.com}
\vskip 10pt
{\bf Shoei Takahashi}\\
{\smallit Keimei Gakuin Junior and High School, Kobe City, Japan}\\
{\tt laptop.syouei@gmail.com}
\vskip 10pt
{\bf Aoi Murakami}\\
{\smallit Keimei Gakuin Junior and High School, Kobe City, Japan}\\
{\tt atatpj728786.55986@gmail.com}
\vskip 10pt
{\bf Akito Tsujii}\\
{\smallit Keimei Gakuin Junior and High School, Kobe City, Japan}\\
{\tt urakihebanam@gmail.com}
\vskip 10pt
{\bf Hikaru Manabe}\\
{\smallit Keimei Gakuin Junior and High School, Kobe City, Japan}\\
{\tt urakihebanam@gmail.com}

\end{center}
\vskip 20pt
\centerline{\smallit Received: , Revised: , Accepted: , Published: } 
\vskip 30pt


\centerline{\bf Abstract}
\noindent
In this paper, we consider impartial and partizan restricted chocolate bar games.
In impartial restricted chocolate bar games, players cut a chocolate bar into two pieces along any horizontal or vertical line and eat  whichever piece is smaller. If the two pieces are the same size, a player can eat either one.

In constrast, partizan restricted chocolate bar games include players designated as Left and Right and chocolate bars with black and white stripes. Left cuts the chocolate bar in two as above and eats the part with fewer black blocks. 
Similarly, Right cuts the bar and eats the part with fewer white blocks. A player loses when they cannot eat the remaining chocolate bar. We provide formulas that describe the winning positions of the previous player, Right, and Left players. We also present an interesting similarity in the graphs of previous players' winning positions for impartial and partizan chocolate bar  games.

\pagestyle{myheadings} 
\markright{\smalltt   (  )\hfill} 
\thispagestyle{empty} 
\baselineskip=12.875pt 
\vskip 30pt


\section{Introduction}
The classic game of Nim is played with stones arranged in discrete heaps or piles. A player can remove any number of stones from any one pile during their turn, and the player who takes the last stone is considered the winner. 

 Given its simplicity and importance in combinatorial game theory, many variant forms of Nim have been devised.

 In Section \ref{maxnim}, we discuss a version of Maximum Nim with an upper bound $f(n)$ on the number of stones that can be removed in terms of the number of stones $n$ in a given pile (see \cite{levinenim}). In this study, we consider the case in which $f(n)= \left\lfloor \frac{n}{2}\right\rfloor $. 
 
 In Section \ref{chocosize}, we explore variants of the game Chomp \cite{gale}. Robin presented the first two-dimensional chocolate bar which used  a rectangular array of squares with a single black square representing a bitter piece \cite{robin}. The two  players take turns breaking the bar into two pieces along a straight line representing the grooves in a chocolate bar
  and eat the piece without the bitter part. The player that leaves their opponent with the single ``bitter" or black block wins. 
This is combinatorially equivalent to game of Nim with two heaps; for example, the chocolate bar game shown in Figure \ref{chocoof5and3} is eqivalent to the classical Nim game with two piles shown in Figure \ref{nimof5and3}.
We have investigated many variants of such chocolate bar games in prior works  \cite{integer2015}, \cite{integers2020}, 
including some with two-dimensional chocolate bars with irregular shapes (Figure \ref{choco2511b}).
We have also provided some for three-dimensional chocolate bars. See \cite{integer2021}.

\begin{minipage}[t]{0.22\textwidth}
\begin{center}
\begin{figure}[H]
\includegraphics[height=0.9cm]{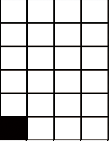}
\caption{Original chocolate bar\\ with a bitter part}\label{chocoof5and3}
\end{figure}
\end{center}
\end{minipage}
\hfill
\begin{minipage}[t]{0.35\textwidth}
\begin{center}
\begin{figure}[H]
\includegraphics[height=0.8cm]{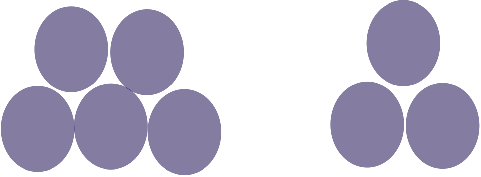}
\caption{Traditional \\ Nim with  two piles}\label{nimof5and3}
\end{figure}
\end{center}
\end{minipage}
\begin{minipage}[t]{0.35\textwidth}
\begin{center}
\begin{figure}[H]
\includegraphics[height=0.9cm]{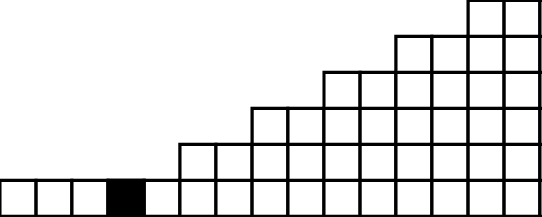}
\caption{Chocolate bar \\with a bitter part, \\ where some squares \\ are removed }\label{choco2511b}
\end{figure}

\end{center}
\end{minipage}

In the present work, we consider a rectangular chocolate bar game with restrictions on the size of the piece to be eaten without a bitter part (Figure \ref{choco2511c}).

Although we have described some characteristics of impartial restricted chocolate bar games in a prior work, we 
present a simplified version of the theory here \cite{thaij2023a}.

These results on impartial restricted chocolate bar games are relevant to our investigation of partizan restricted chocolate bar  games owing to some interesting similarities and differences of note.

\begin{figure}[H]
\begin{center}
\includegraphics[height=1.cm]{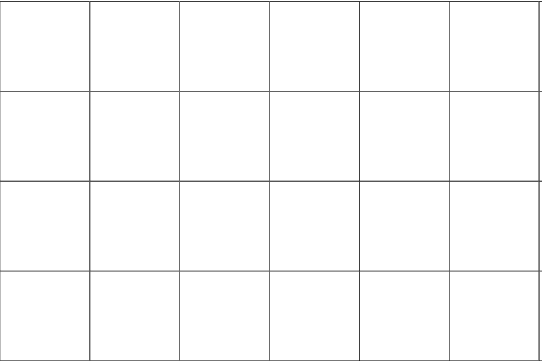}
\caption{Chocolate bar without any bitter part}\label{choco2511c}
\end{center}
\end{figure}

Given this restriction, we consider a chocolate bar game equivalent to the sum of two Maximum Nim. Hence, results that have been established for Maximum Nim can be applied to this chocolate bar game.

In Section \ref{partizan1}, we explore some characteristics of partizan restricted chocolate bar games as a natural generalization of impartial restricted games.

These involve two players designated as Left and Right, with chocolate bars with black and white stripes. Left cuts the  bar  into two pieces and eats that with fewer black blocks. Similarly, Right cuts the bar and eats the piece with fewer white blocks. 
A player loses when they cannot eat the remaining piece. We provide formulas that describe the previous, Right, and Left players' winning positions. In particular, we discuss an interesting similarity in the graphs of previous players' winning positions for impartial and partizan chocolate bar  games.


We briefly review some necessary concepts in combinatorial game theory (see \cite{lesson} for more details). 
We denote the sets of non-negative integers and natural numbers as $\mathbb{Z}_{\ge 0}$ and  $\mathbb{N}$, respectively.
Here, we consider impartial games with no draws, which admit of only two outcome classes. 

$(a)$ A position is called a $\mathcal{P}$-\textit{position} if it is a winning position for the previous player (the player who just moved), as long as he/she plays correctly at every stage.\\
$(b)$ A position is called an $\mathcal{N}$-\textit{position} if it is a winning position for the next player, as long as he/she plays correctly at every stage.

In combinatorial game theory, the Nim-sum  indicates addition without carryovers for numbers expressed in base $2$, which is identical to the logical operation of exclusive or.

Let $x$, $y$ $\in Z_{\ge 0}$. We represent $x = \sum_{i=0}^n x_i 2^i$ and $y = \sum_{i=0}^n y_i 2^i$ with $x_i,y_i \in \{0,1\}$.
We define the Nim-sum $x \oplus y$ as
$x \oplus y = \sum\limits_{i = 0}^n {{w_i}} {2^i}, \text{ where }w_{i}=x_{i}+y_{i} \ ( \mathrm{mod} \ 2).$

The Grundy number is another important tool in combinatorial game theory. To express this for the game considered here, we need to define the ``move" that the players make on successive turns.\\
$(i)$ For any position $\mathbf{p}$ of a game $\mathbf{G}$, there exists a set of positions that can be reached by making precisely one move in $\mathbf{G}$, which is denoted by \textit{move}$(\mathbf{p})$. \\
$(ii)$ The \textit{minimum excluded value} ($\textit{mex}$) of a set $S$ of non-negative integers is the least non-negative integer, not in S. \\
$(iii)$ Let $\mathbf{p}$ be the current position in an impartial game. The associated Grundy number is denoted by $\mathcal{G}(\mathbf{p})$ and is
recursively defined by $\mathcal{G}(\mathbf{p}) = \textit{mex}\{\mathcal{G}(\mathbf{h}): \mathbf{h} \in move(\mathbf{p})\}.$

The Grundy number of a position is important for two reasons. It is related to the $\mathcal{P}$-position presented in Theorem \ref{theoremofsumga}, and can also applied to the sum of multiple games.

\begin{theorem}\label{theoremofsumga}
Let $\mathcal{G}$ represent the Grundy number of the combinatorial game $\mathbf{G}$.
Then, for any position $\mathbf{p}$ of $\mathbf{G}$, we have
$\mathcal{G}(\mathbf{p})=0$ if and only if $\mathbf{p}$ represents a $\mathcal{P}$-position.
\end{theorem}
For the proof, please see \cite{lesson}.

The sum of two games is an important concept in combinatorial game theory.

If $\mathbf{G_{1}}$ and $\mathbf{G_{2}}$ are combinatorial games, their sum, denoted by $\mathbf{G_{1}}+\mathbf{G_{2}}$,
represents a game in which each player alternately chooses one among $\mathbf{G_{1}}$ and $\mathbf{G_{2}}$, 
and plays a move in the selected game; the player who cannot play loses the game. This scenario occurs only when $\mathbf{G_{1}}$ and $\mathbf{G_{2}}$ have reached a terminal position.

There is a simple relation for the Grundy number of the sum of two games.

Let $\mathcal{G}_{1}$ and $\mathcal{G}_{2}$ represent the Grundy numbers of $\mathbf{G_{1}}$ and $\mathbf{G_{2}}$.
Then, the Grundy number of a position $\{\mathbf{g},\mathbf{h}\}$ in the game $\mathbf{G_{1}}+\mathbf{G_{2}}$ is
$\mathcal{G}_{1}(\mathbf{g})\oplus \mathcal{G}_{1}(\mathbf{h})$.

Given this background,
the authors present new results from this point forward.

\section{Maximum Nim}\label{maxnim}   
We consider maximum Nim as follows.

Suppose there is a pile of $n$ stones, and two players take turns removing stones from the pile.
At each turn, the player is allowed to remove at least one and at most 
$f(m)$ stones if the number of remaining stones is $m$. The player who removes the last stone or stones is the winner. 
Here, $f(m)$ represents a function whose values are non-negative integers for $m \in Z_{\ge 0}$ such that 
$0 \leq f(m) -f(m-1) \leq 1$ for any natural number $m$. We refer to $f$ as the rule function.

\begin{lemma}\label{lemmabylevinenim}
Let $\mathcal{G}$ represent the Grundy number of the maximum Nim with the rule function $f(x)$. Then, we have the following properties:\\
$(i)$ If $f(x) = f(x-1)$, $\mathcal{G}(x) = \mathcal{G}(x-f(x)-1)$.\\
$(ii)$ If $f(x) > f(x-1)$, $\mathcal{G}(x) = f(x)$.
\end{lemma}
This is Lemma 2.1 of \cite{levinenim}.

In this paper, we assume that $f(m)= \left\lfloor \frac{m}{2}\right\rfloor $, where $ \left\lfloor x \right\rfloor$ represent the greatest integer less than or equal to $x$ for
$x \in \mathbb{Z}_{\ge 0}$.

For the maximum Nim of $x$ stones with rule function $f(x)= \left\lfloor \frac{x}{2}\right\rfloor $, \\
$\textit{move}(x)$ $= \{x-u:1 \leq u \leq \left\lfloor \frac{x}{2}\right\rfloor \text{ and } u \in \mathbb{N} \}$.

\begin{lemma}\label{grundylemmaformax}
Let $\mathcal{G}$ represent the Grundy number of the maximum Nim with the rule function $f(m)= \left\lfloor \frac{m}{2}\right\rfloor $. Then, we have the following properties:\\
$(i)$ For an even number $x$ such that $x=2m$ for $m \in \mathbb{Z}_{\ge 0}$, we have 
\begin{equation}
\mathcal{G}(x)  = m. \nonumber
\end{equation}
$(ii)$
For an odd number $x$ such that $x=2^k(2m+1)-1$ for $k \in \mathbb{N}$ and $m \in \mathbb{Z}_{\ge 0}$, we have 
\begin{equation}
\mathcal{G}(x) = m \nonumber.
\end{equation}
\end{lemma}
\begin{proof}
$\mathrm{(i)}$
If $x=2m$ for $m \in \mathbb{Z}_{\ge 0}$, $f(x) > f(x-1)$. Therefore, by $(ii)$ of Lemma \ref{lemmabylevinenim}, 
\begin{equation}
\mathcal{G}(2m) = f(2m)= \left\lfloor \frac{2m}{2}\right\rfloor = m. \label{evencase} 
\end{equation}
$\mathrm{(ii)}$ By $(\ref{evencase})$ and $(i)$ of Lemma \ref{lemmabylevinenim}, we have
for $k \in \mathbb{N}$ and $m \in \mathbb{Z}_{\ge 0}$
\begin{align}
 \mathcal{G}(x) & = \mathcal{G}(2^{k}(2m+1)-1) \nonumber \\
& = \mathcal{G}(2^{k}(2m+1)-1- \left\lfloor \frac{2^{k}(2m+1)-1}{2}\right\rfloor -1)  \nonumber \\
& = \mathcal{G}(2^{k}(2m+1)-1- (2^{k-1}(2m+1)-1)) -1)   \nonumber \\
& = \mathcal{G}(2^{k-1}(2m+1)-1)  \nonumber \\
&  \hspace{1cm} \vdots \nonumber \\
& = \mathcal{G}((2m+1)-1) = \mathcal{G}(2m)=m.  \nonumber 
\end{align}	
\end{proof}

\begin{lemma}\label{lemmaforgrundymax}
For any number $n \in \mathbb{Z}_{\ge 0}$
\begin{equation}
\{x:\mathcal{G}(x)=n\}=\{2^k(2n+1)-1:k \in \mathbb{Z}_{\ge 0}\}.\label{grundyset1}
\end{equation}
\end{lemma}
\begin{proof}
Let $n \in \mathbb{Z}_{\ge 0}$. If $k=0$, $2^k(2n+1)-1=2n$. Then, by $(i)$ of Lemma \ref{grundylemmaformax} we have 
$\mathcal{G}(2n)=n$. If $k \geq 1$, by $(ii)$ of Lemma \ref{grundylemmaformax}, 
$\mathcal{G}(2^k(2n+1)-1) = n$.
Therefore, 
\begin{equation}
\{x:\mathcal{G}(x)=n\} \supset \{2^k(2n+1)-1:k \in \mathbb{Z}_{\ge 0}\}.\label{grundyset2}
\end{equation}
Suppose that $\mathcal{G}(x)=n$ for $n \in \mathbb{N}$. 
By Lemma \ref{grundylemmaformax}, we have
$x = 2n$ or
$x=2^k(2n+1)-1$ for some $k \in \mathbb{N}$. Therefore we have
\begin{equation}
\{x:\mathcal{G}(x)=n\} \subset \{2^k(2n+1)-1:k \in \mathbb{Z}_{\ge 0}\}.\label{grundyset3}
\end{equation}
By (\ref{grundyset2}) and (\ref{grundyset3}) we have (\ref{grundyset1}).
\end{proof}

\section{Chocolate bar games with restrictions on the size of chocolate bar  to be eaten}\label{chocosize}

We define chocolate bar games with restrictions.
\begin{defn}\label{definitionofmaxchoco}
$(i)$ A chocolate bar is a rectangular array of squares. \\
$(ii)$ We describe the position of this chocolate bar as $(x,y)$ when $x$ and $y$ represent the height and the width of the chocolate.
Then, the size of the chocolate bar is calculated by $xy$. \\
$(iii)$ Two players take turns and break the chocolate bar along any one of the horizontal or vertical lines into two pieces and eat the smaller one of the pieces, where players can eat either piece if the pieces are of equal size.\\
$(iv)$ The player who is left with a $1 \times 1$ piece of chocolate and hence cannot make another move loses the game.  
Therefore the $1 \times 1$ piece of chocolate with the coordinates $(1,1)$ is the terminal position of the game.
\end{defn}

For an example of the position of a chocolate bar, see Figure \ref{choco46}.

\begin{remark}
Note that cutting chocolate bar horizontally is the same as reducing the first coordinate, and
cutting chocolate bar  vertically is the same as reducing the second coordinate.
\end{remark}

\begin{theorem}\label{grundyforchoco}
Let $\mathcal{G}(x,y)$ be the Grundy number of the chocolate bar game $(x,y)$,
$\mathcal{G}(x)$ and $\mathcal{G}(y)$ be the Grundy number of the Max Nim of 
$x$-stones and $y$-stones with the rule sequence $f(n)= \left\lfloor \frac{n}{2}\right\rfloor $.
If we define 
$G_n=\{2^k(2n+1)-1:k \in \mathbb{Z}_{\ge 0}\}$ for $n \in  \mathbb{Z}_{\geq0}$, then 
we have for $n \in  \mathbb{Z}_{\geq0}$,
\begin{equation}
\{(x,y):\mathcal{G}(x,y)=n\}=\cup\{G_s \cap G_t :s,t \in  \mathbb{Z}_{\geq0} \text{ such that } s \oplus t = n\}.\label{g1g2sum}
\end{equation}
\end{theorem}
\begin{proof}
Since the chocolate bar  game is the sum of two Max Nim with  a rule sequence $f(n)= \left\lfloor \frac{n}{2}\right\rfloor$, we have
$\mathcal{G}(x,y) = \mathcal{G}(x) \oplus \mathcal{G}(y)$. Therefore, by Lemma \ref{lemmaforgrundymax}, we have $(\ref{g1g2sum})$.
\end{proof}

\begin{defn}\label{ppmaxpp}
Let 
\begin{align}
\mathcal{P} = &\{(x,y):y=2^kx+2^k-1, x \in \mathbb{N} \text{ and } k \in  \mathbb{Z}_{\geq0}\} \nonumber \\
& \cup \{(x,y):x=2^ky+2^k-1, y\in \mathbb{N} \text{ and } k \in  \mathbb{Z}_{\geq0}\}. \nonumber
\end{align}	
\end{defn}

\begin{theorem}\label{ppositionmax}
The set $\mathcal{P}$ in Definition \ref{ppmaxpp} is the set of 
$\mathcal{P}$-positions of the chocolate bar game of Definition \ref{definitionofmaxchoco}.
\end{theorem}
\begin{proof}
Since the chocolate bar  game is the sum of two Max Nim, 
$(x,y) \in \mathcal{P}$ if and only if
\begin{equation}
\mathcal{G}(x) = \mathcal{G}(y) = n \label{gxeqgy}   
\end{equation}
for some $n \in  \mathbb{Z}_{\geq0}$.
By Lemma \ref{grundylemmaformax}, 
we have (\ref{gxeqgy}) if and only if
\begin{equation}
(x+1,y+1) = (2^s(2n+1),2^t(2n+1)) \label{relationxy1}
\end{equation}
for $s,t,n \in  \mathbb{Z}_{\geq0}$.
Equation (\ref{relationxy1}) is equivalent to 
\begin{equation}
y+1 = 2^{s-t}(x+1) \nonumber
\end{equation}
for $s, t \in  \mathbb{Z}_{\geq0}$ such that $s \geq t$
or
\begin{equation}
x+1 = 2^{t-s}(y+1) \nonumber
\end{equation}
for $s,t \in  \mathbb{Z}_{\geq0}$ such that $t \geq s$.
Since we need to consider only the chocolate bar whose 
 height and width are natural numbers, and we assume that
 $x, y \in \mathbb{N}$. Therefore, the set $\mathcal{P}$ in Definition \ref{ppmaxpp} is 
 the set of 
$\mathcal{P}$-positions of the chocolate bar game of Definition \ref{definitionofmaxchoco}.
\end{proof}

The chocolate bar game in this section was presented at the 17th International Olympiad in Informatics \cite{IOI2005}, where the set of $\mathcal{P}$-positions was presented with proof. Still, Grundy numbers of the chocolate bar games were not mentioned. 
In \cite{stackex}, the set of $\mathcal{P}$-positions was presented with a proof that is different from the one used in \cite{IOI2005}.

Each point in the graph in Figure \ref{maxp} with position $(x,y)$ (the bottom left is $(1,1)$) represents a $x \times y$ rectangular chocolate that is a $\mathcal{P}$-position.
The $\mathcal{P}$-positions in this graph of the chocolate bar game appear to be arranged in lines. For example, the central diagonal has a slope of 1; the following line down has a slope of $\frac{1}{2}$; the next, $\frac{1}{4}$; and so on.

\begin{figure}[H]
\begin{minipage}[t]{0.25\textwidth}
\begin{center}
\includegraphics[height=2.5cm]{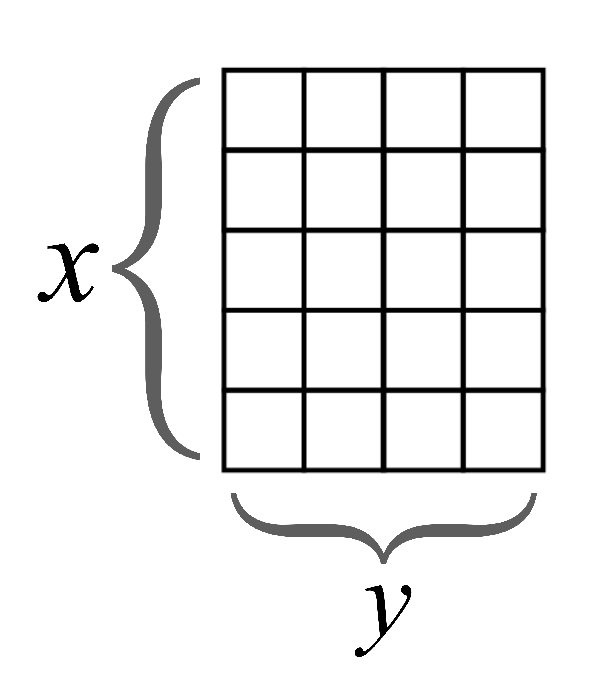}
\caption{(5,4)}
\label{choco46}
\end{center}
\end{minipage}
\begin{minipage}[t]{0.45\textwidth}
\begin{center}
\includegraphics[height=3.5cm]{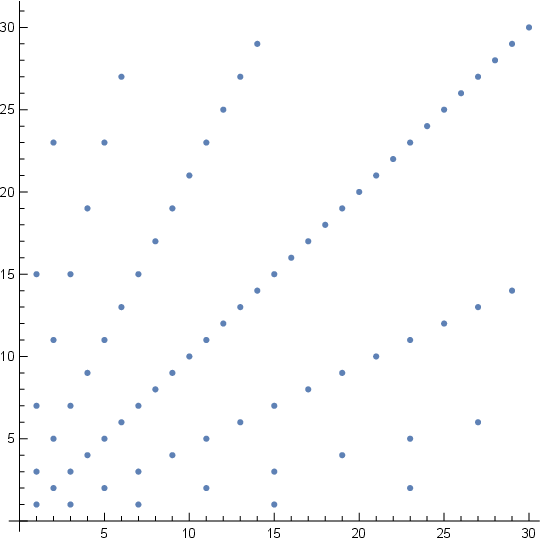}
\caption{P-positions}
\label{maxp}
\end{center}
\end{minipage}
\end{figure}

\section{Partizan chocolate bar  game with a restriction on the size of chocolate bar  to be eaten }\label{partizan1}
There are two different classifications of a game in combinatorial games: impartial
and partizan. An impartial game is one where the same moves exist for both players and games in 
 sections  \ref{maxnim}  and \ref{chocosize} are impartial.

In this section, we study partizan chocolate games, where  players control their own moves. 

Let the name of two players be Left (using she as a pronoun) and Right (using he as a pronoun). Here we use chocolate bars with black and white stripes. See Figures \ref{bwchocolate1}, \ref{bwchocolate2}, \ref{1by5}, and \ref{5by1} for examples. The use of this black and white chocolate bar  was suggested by Prof. R. J. Nowakowski when one of the authors discussed a partizan version of chocolate bar games at Combinatorial Game Theory Colloquium IV.

Left cuts the chocolate bar  into two parts, and she eats the part with fewer black blocks. When the number of black blocks is the same for both parts, she can eat either. When there is only a one-by-one white chocolate bar that is $(1,1,0)$ in figures \ref{1by1} and \ref{m1by1}, she can eat it. 

Similarly, Right cuts the chocolate bar  into two and eats the part with fewer white blocks. If the number of white blocks is the same, he can eat either. When there is only a one-by-one black chocolate bar that is $(1,1,1)$ in figures \ref{1by1} and \ref{m1by1}, he can eat it. 

A player loses in the game when she or he cannot eat the remaining chocolate bar.

\begin{minipage}[t]{0.2\textwidth}
\begin{center}
\begin{figure}[H]
\includegraphics[height=1.7cm]{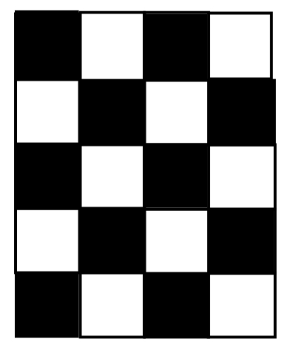}
\caption{\\(5,4,1)}\label{bwchocolate1}
\end{figure}
\end{center}
\end{minipage}
\begin{minipage}[t]{0.2\textwidth}
\begin{center}
\begin{figure}[H]
\includegraphics[height=1.7cm]{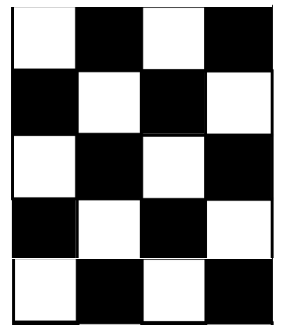}
\caption{\\(5,4,0)}\label{bwchocolate2} 
\end{figure}
\end{center}
\end{minipage}
\begin{minipage}[t]{0.2\textwidth}
\begin{center}
\begin{figure}[H]
\includegraphics[height=1.7cm]{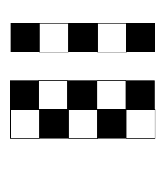}
\caption{\\(1,5,1)\\and(2,5,1)}\label{1by5}
\end{figure}
\end{center}
\end{minipage}
\begin{minipage}[t]{0.2\textwidth}
\begin{center}
\begin{figure}[H]
\includegraphics[height=1.7cm]{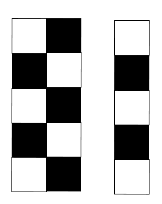}
\caption{\\(5,2,0)\\and(5,1,0)}\label{5by1}
\end{figure}
\end{center}
\end{minipage}
\begin{minipage}[t]{0.1\textwidth}
\begin{center}
\begin{figure}[H]
\includegraphics[height=1.7cm]{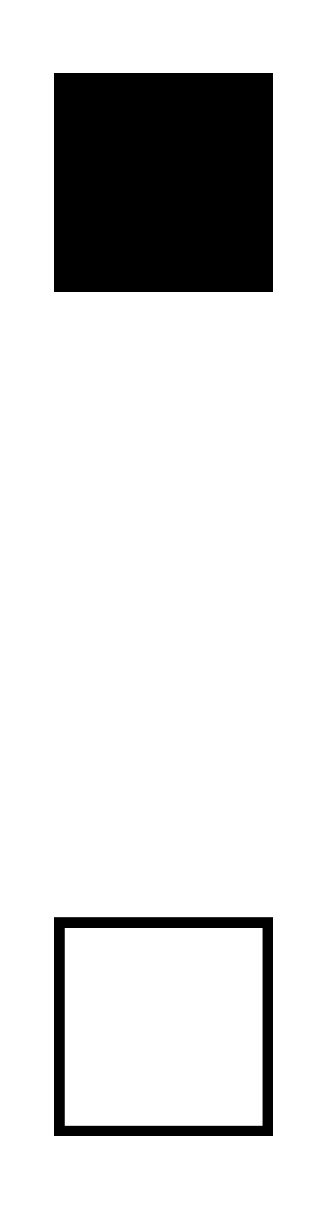}
\caption{(1,1,1)\\and(1,1,0)}\label{1by1}
\end{figure}
\end{center}
\end{minipage}

In the following, we describe a bar of black and white chocolate bar  with a matrix of numbers, where $1$ is for a black block and $0$ is for a white block.

Black and white chocolate bars in figures \ref{bwchocolate1}, \ref{bwchocolate2}, \ref{1by5}, and \ref{5by1} are represented by matrices
in figures \ref{matrix1}, \ref{matrix0}, \ref{matrix2}, and \ref{matrix3}.

\begin{minipage}[t]{0.2\textwidth}
\begin{center}
\begin{figure}[H]
\includegraphics[height=1.6cm]{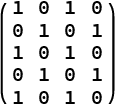}
\caption{\\(5,4,1)}\label{matrix1}
\end{figure}
\end{center}
\end{minipage}
\begin{minipage}[t]{0.2\textwidth}
\begin{center}
\begin{figure}[H]
\includegraphics[height=1.6cm]{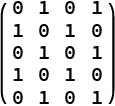}
\caption{\\(5,4,0)}\label{matrix0}
\end{figure}
\end{center}
\end{minipage}
\begin{minipage}[t]{0.2\textwidth}
\begin{center}
\begin{figure}[H]
\includegraphics[height=1.65cm]{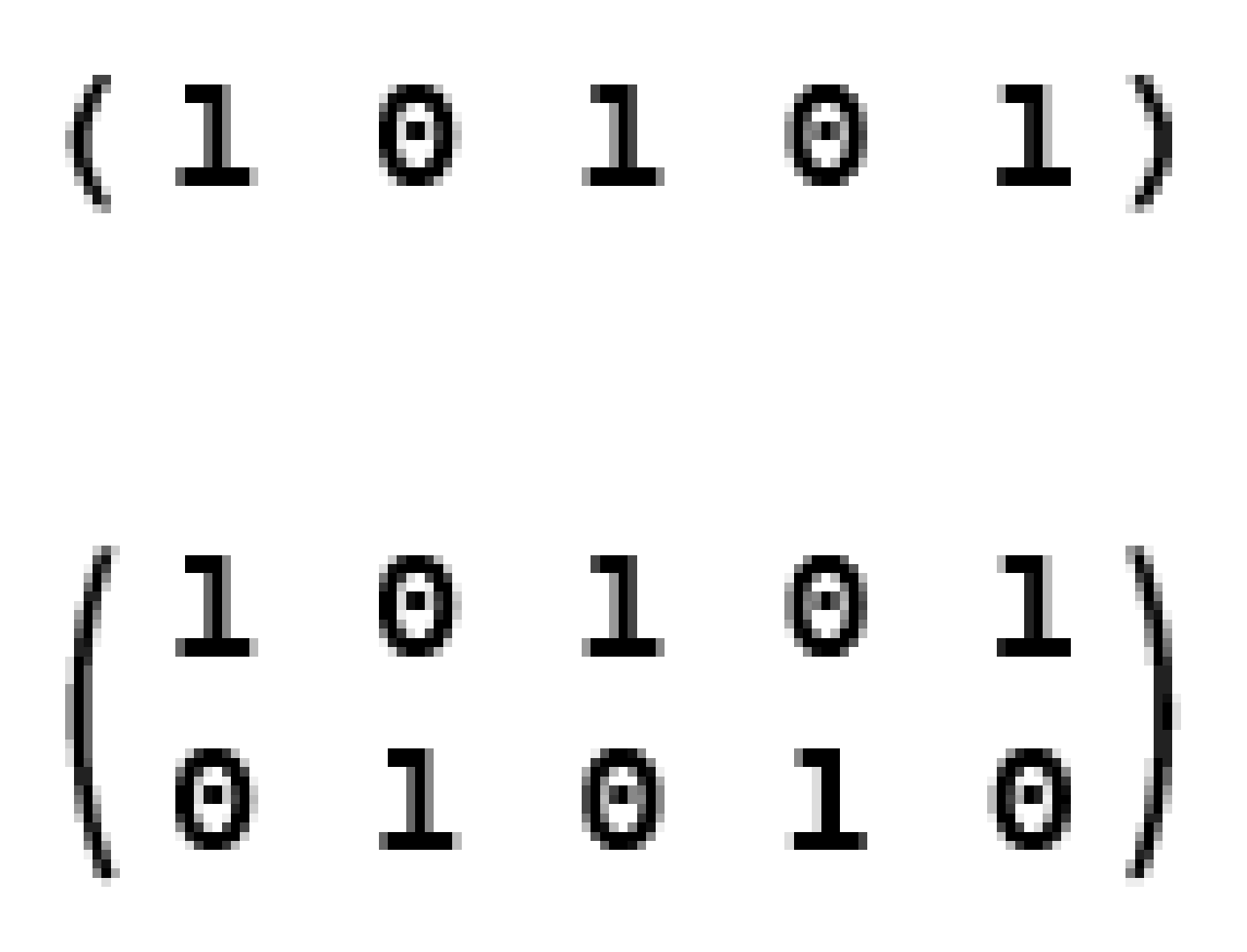}
\caption{\\(1,5,1)\\ and (2,5,1)}\label{matrix2}
\end{figure}
\end{center}
\end{minipage}
\begin{minipage}[t]{0.2\textwidth}
\begin{center}
\begin{figure}[H]
\includegraphics[height=1.7cm]{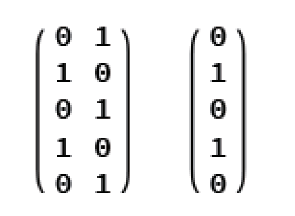}
\caption{\\(5,2,0)\\ and (5,1,0)}\label{matrix3}
\end{figure}
\end{center}
\end{minipage}
\begin{minipage}[t]{0.1\textwidth}
\begin{center}
\begin{figure}[H]
\includegraphics[height=1.5cm]{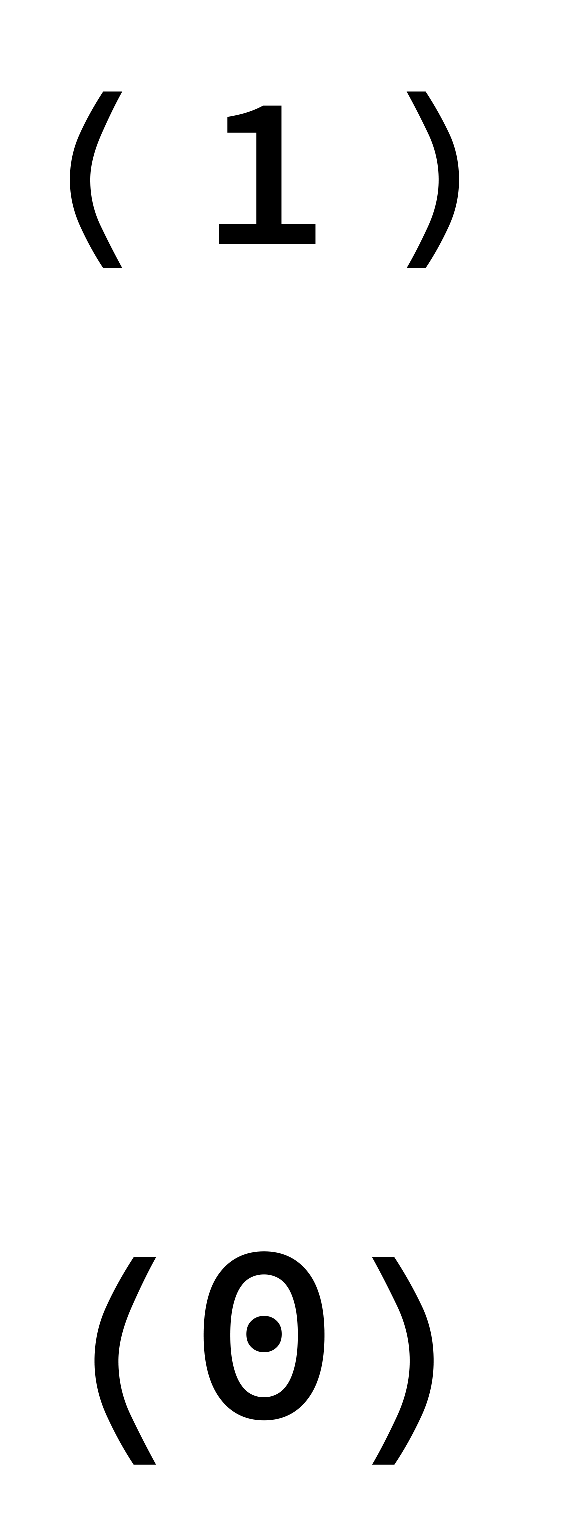}
\caption{(1,1,1)\\and(1,1,0)}\label{m1by1}
\end{figure}
\end{center}
\end{minipage}

We denote a chocolate bar by $(x,y,s)$, where $x$ and $y$ are the height and the width of the chocolate bar, respectively, and $s$ is the first number of the first row.

Please compare the chocolate bar  in Figure \ref{bwchocolate1} and the chocolate bar  in Figure \ref{bwchocolate2}. These are represented as matrices in \ref{matrix1} and \ref{matrix0}. Note that the difference between $(5,4,1)$ and $(5,4,0)$ is the value of the third coordinate, the first number of the first row of each matrix.

According to the rule of the game, there are three ways that the game ends.
A player loses the game when there is no chocolate bar  left.  Player L loses when there is only $(1,1,1)$, and Player R loses when there is only $(1,1,0).$

Beside $\mathcal{P}$-\textit{position} and  $\mathcal{N}$-\textit{position} in 
Section \ref{maxnim}, we need the following two outcome classes in Section \ref{partizan1}.

\begin{defn}\label{LRpositions}
$(a)$ A position is referred to as a $\mathcal{L}$-\textit{position} if it is a winning position for L.\\
$(b)$ A position is referred to as an $\mathcal{R}$-\textit{position} if it is a winning position for R.
\end{defn}

\begin{theorem}
Every position of a game belongs to exactly one of  four outcome classes 
 $\mathcal{P}$-position, $\mathcal{N}$-position, $\mathcal{L}$-position, and $\mathcal{R}$-position.
\end{theorem}
\begin{proof}
This fact is presented on page 36 of \cite{lesson}. 
\end{proof}

\begin{lemma}\label{howtodecide}
$(i)$ A position is  a $\mathcal{P}$-position if and only if 
$R$ always move to a  $\mathcal{L}$ or $\mathcal{N}$-position and 
$L$ always moves to a $ \mathcal{R}$ or $\mathcal{N}$-position.\\
$(ii)$ A position is  a $\mathcal{N}$-position if and only if $R$ can move to a  $R$ or $\mathcal{P}$-position and $L$ can move to a $\mathcal{L}$ or $ \mathcal{P}$-position.\\
$(iii)$ A position is  a $\mathcal{L}$-position if and only if $R$ always move to a  $\mathcal{L}$ or $\mathcal{N}$-position and $L$ can move to a $\mathcal{L}$ or $ \mathcal{P}$-position.\\
$(iv)$  A position is  a $\mathcal{R}$-position if and only if 
$R$ can move to a  $R$ or $\mathcal{P}$-position and $L$ always moves to a $ \mathcal{R}$ or $\mathcal{N}$-position.
\end{lemma}
\begin{proof}
This fact is presented as Observation 2.4 on page 37 of \cite{lesson}. 

\end{proof}
The result of Lemma \ref{howtodecide} is presented as
 Table \ref{decideoutcome}. By this, we can decide the outcome class of each position.

\begin{table}[H]
\centering
\caption{\label{gamevalue2}How to decide the outcome class of a game}\label{decideoutcome}
\begin{tabular}{|l|l|l|} 
\hline
&  $R$ can move to a & $R$ always moves to  \\ 
&  $ \mathcal{R}$ or $\mathcal{P}$-position  & a $\mathcal{L}$ or $\mathcal{N}$-position \\ 
\hline
$L$ can move to a &  &   \\ 
 $\mathcal{L}$ or $ \mathcal{P}$-position & \hspace{1cm} $\mathcal{N}$-position & \hspace{1cm} $\mathcal{L}$-position  \\ 
\hline
$L$ always moves to &  &    \\ 
a $ \mathcal{R}$ or $\mathcal{N}$-position & \hspace{1cm} $\mathcal{R}$-position & \hspace{1cm} $\mathcal{P}$-position 
\\ 
\hline
\end{tabular}
\end{table}

\begin{lemma}\label{allnatural}
Let 
\begin{equation}
\mathbb{N}_1 = \{2^s(2p+5)-1:s,p \in  \mathbb{Z}_{\geq0}\},\nonumber
\end{equation}
\begin{equation}
\mathbb{N}_2 =  \{2^s 4-1:s \in  \mathbb{Z}_{\geq0}\},\nonumber
\end{equation}
and 
\begin{equation}
\mathbb{N}_3 =  \{2^s 3-1:s \in  \mathbb{Z}_{\geq0}\}.\nonumber
\end{equation}
Then the set of natural numbers
\begin{equation}
\mathbb{N} = \{1\} \cup \mathbb{N}_1 \cup \mathbb{N}_2 \cup \mathbb{N}_3.\nonumber
\end{equation}
\end{lemma}
\begin{proof}
Let $n$ be a natural number such that 
\begin{equation}
n \geq 2. \label{biggerthan3}
\end{equation}
For $n$, there exist $m \in  \mathbb{Z}_{\geq0}$ and $k \in \mathbb{N}$ such that $n+1=2^m k$ and $k$ is odd.
If $k=1$, then by (\ref{biggerthan3}), we have $2^m =n+1 \geq 3$. Therefore, we have $m \geq 2$, and 
 $n=2^{m-2}4-1 \in \mathbb{N}_2$.
If $k=3$, $n = 2^m3-1 \in \mathbb{N}_3$. If $k$ is an odd number such that 
$k \geq 5$, then $n = 2^m(2p+5)-1 \in \mathbb{N}_1$ for some $p \in  \mathbb{Z}_{\geq0}$.
\end{proof}

\begin{figure}[H]
\begin{center}
\includegraphics[height=2.8cm]{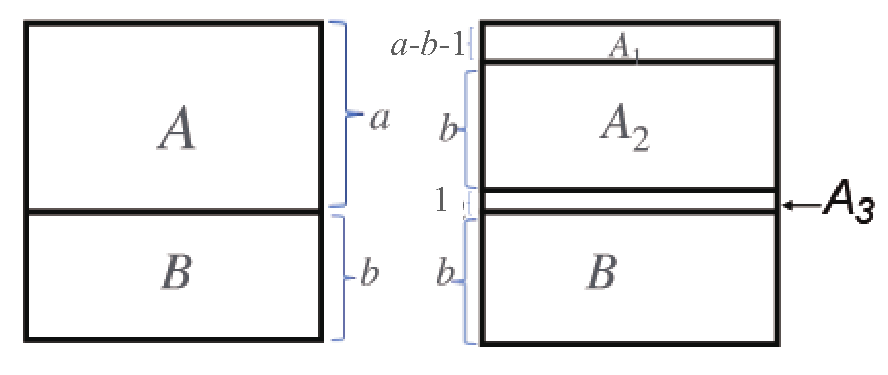}
\caption{To divide chocolate bars}\label{cutchoco}
\end{center}
\end{figure}

\begin{lemma}\label{morethanhalf}
Suppose that  we cut a chocolate bar  $(a+b,m,p)$ into two pieces $A=(a,m,s)$ and $B=(b,m,t)$, where $a,b,m \in \mathbb{N},$ $m \geq 2$, $a>b\geq 1$ and $s,t =0$ or $1$.
Then, the number of $0$ and the number of $1$ in $A$ are larger than those of $B$, respectively.
\end{lemma}
\begin{proof}
We divide $A$ into $A_1=(a-b-1,m,p)$, $A_2=(b,m,r)$ and $A_3=(1,m,s)$, where $r,s =0,1$. See Figure \ref{cutchoco}.
If $a=b+1$, then $A_1=(a-b-1,m,p)=(0,m,p) = \emptyset$.
Since $B$ and $A_2$ are horizontally symmetrical as matrices of $0$ and $1$, 
the numbers of $0$ and $1$ are the same for both.
Since $m \geq 2,$ $A_3$ contains some $0$ and $1$, 
 the number of $0$ and the number of $1$ in $A=A_1 \cup A_2 \cup A_3$ are larger than those of $A_2$.
\end{proof}

\begin{remark}\label{remarkbigsmall}
Suppose you have a chocolate bar  $(x,y,s)$ with $x,y \geq 2$ and $s=0,1$.
When you cut this chocolate bar  into two pieces, then 
by Lemma \ref{morethanhalf}, you can eat the smaller one and cannot eat the bigger one. If two pieces are the same size, you have to count the $0$ and $1$ in these pieces to decide which to eat.
As for the chocolate bars $(n,1,s)$ or $(1,n,s)$ with 
\end{remark}

\begin{lemma}\label{lemmaforheight1}
Let $n \in \mathbb{Z}_{\geq0}$. Then, we have the following: \\
$(i)$ If $L$ cuts $(4n+1,1,0)$ chocolate bar , she gets chocolate bars
$(2n,1,0),\cdots, (4n,1,0)$.\\
$(ii)$ If $R$ cuts $(4n+1,1,0)$ chocolate bar , he gets chocolate bars
$(2n+1,1,0), \cdots, (4n,1,0)$.\\
$(iii)$ If $R$ cuts $(4n+3,1,0)$ chocolate bar , he gets chocolate bars
$(2n+1,1,0), \cdots, (4n+2,1,0)$.
\end{lemma}

\begin{proof}
$\mathrm{(i)}$ Let $A=(4n+1,1,0)$, $A_1=(2n,1,0)$, and $A_2=(2n+1,1,0)$. Then $A=A_1 \cup A_2$ and the number of $1$ in $A_1$ is the same as the number of $1$ in $A_2$, and hence you can eat $A_2$ to get $A_1$. It is clear that $A_1$ is the smallest chocolate bar  you can get from $A$.\\
$\mathrm{(ii)}$ 
Let $A=(4n+1,1,0)$, $A_1=(2n+1,1,0)$, $A_2=(2n,1,0)$
, $A^{\prime}_1=(2n,1,0)$, $A^{\prime}_2=(2n+1,1,0)$. Then 
we have $A=A_1 \cup A_2$ and $A=A^{\prime}_1 \cup A^{\prime}_2$.
The number of $0$ in $A_1$ is $n+1$, the number of $0$ in $A_2$ is $n$, 
the number of $0$ in $A^{\prime}_1$ is $n$, the number of $0$ in $A^{\prime}_2$ is $n+1$, and hence $R$ can eat $A_2$ to get $A_1$, but  $R$ cannot eat $A^{\prime}_2$ to get $A^{\prime}_1$.
It is clear that $A_1$ is the smallest chocolate bar  $R$ can get from $A$.\\
$\mathrm{(iii)}$ 
Let $A=(4n+3,1,0)$, $A_1=(2n+1,1,0)$, and $A_2=(2n+2,1,0)$. Then $A=A_1 \cup A_2$, and the number of $0$ in $A_1$ is $n+1$, and the number of $0$ in $A_2$ is $n+1$.
Therefore, $R$ can eat $A_2$ to get $A_1$. Clearly, $A_1$ is the smallest chocolate bar $R$ can get from $A$.
\end{proof}

\begin{defn}\label{partizanpposition}
Let $\mathcal{P}_{a}=\{(2^n(2p+5)-1, 2^m(2p+5)-1,s):n,m \in \mathbb{Z}_{\geq0}, p \in \mathbb{Z}_{\geq0} \text{ and } s =0,1\}$,
$\mathcal{P}_{b}= \{(2^n3-1,2^m4-1,s):n,m \in \mathbb{Z}_{\geq0} \text{ and } s =0,1\}$ and 
$\mathcal{P}_{c}= \{(2^n4-1,2^m3-1,s):n,m \in \mathbb{Z}_{\geq0} \text{ and } s =0,1\}$.
Let $\mathcal{P}=\mathcal{P}_{a} \cup \mathcal{P}_{b}\cup \mathcal{P}_{c}\cup \{(1,2,s),(2,1,s):s =0,1\}\cup \{(0,0)\}$.
\end{defn}

\begin{defn}\label{defofL} 
Let $\mathcal{L}_{1}=\{(2n+5, 1,0):n \in \mathbb{Z}_{\geq0} \}$,
$\mathcal{L}_{2}= \{(1,2n+5,0):n  \in \mathbb{Z}_{\geq0} \}$ and  $\mathcal{L}=\mathcal{L}_{1} \cup \mathcal{L}_{2}\cup \{(1,1,0)\}$.
\end{defn}

\begin{defn}\label{deofR}
Let $\mathcal{R}_{1}=\{(2n+5, 1,1):n \in \mathbb{Z}_{\geq0} \}$,
$\mathcal{R}_{2}= \{(1,2n+5,1):n  \in \mathbb{Z}_{\geq0} \}$ and $\mathcal{R}=\mathcal{R}_{1} \cup \mathcal{R}_{2}\cup \{(1,1,1)\}$.
\end{defn}

\begin{defn}\label{deofN}
Let $\mathcal{N}=\{(x,y,s):x,y \in \mathbb{N}, s =0,1 \text{ and } (x,y,s) \notin \mathcal{P} \cup \mathcal{L} \cup \mathcal{R}\}.$
\end{defn}

\begin{lemma}\label{belongtoN}
We have the following $(i)$, $(ii)$, $(iii)$, and $(iv)$ for the set $\mathcal{N}$:\\
$(i)$  We have   $(3,1,p), (1,3,p)  \in  \mathcal{N}$ for $p=0,1$.\\
$(ii)$ We have $(2n,1,p), (1,2n,p)  \in  \mathcal{N}$ for 
$n \in \mathcal{N}$ such that $n \geq 2$ and $p=0,1$.
\end{lemma}
\begin{proof}
By Definitions \ref{partizanpposition}, \ref{defofL}, \ref{deofR}, and \ref{deofN}
if $(n,1,s) \in \mathcal{L}$ or $\mathcal{R}$, $n=1$ or $n \geq 5$ and $n$ is odd.
If $(n,1,s) \in \mathcal{P}$, $n=2$. Therefore, we have $(i)$ and $(ii)$.
\end{proof}

\begin{lemma}\label{decideLorR}
We have the following $(i)$, $(ii)$, $(iii)$, and $(iv)$:\\
$(i)$  If $L$ starts with a position in $\mathcal{L}$, then she wins or can move to a position in $\mathcal{L} \cup \mathcal{P}$.\\
$(ii)$ If $R$ starts with a position in $\mathcal{R}$, then he wins or can move to a position in $\mathcal{R}\cup \mathcal{P}$.\\
$(iii)$  If $R$ starts with a position in $\mathcal{L}$, then he loses or always move to a position in $\mathcal{L} \cup \mathcal{N}$.\\
$(iv)$  If $L$ starts with a position in $\mathcal{R}$, then he loses or always move to a position in $\mathcal{R} \cup \mathcal{N}$.
\end{lemma}
\begin{proof}
$\mathrm{(i)}$ If $L$ starts with a position $(2n+5, 1,0)$ with $n \in \mathbb{N}$, then she can move to  the position $(2(n-1)+5, 1,0) \in \mathcal{L}$.
If she starts with $(5,1,0)$, then by $(i)$ of Lemma \ref{lemmaforheight1} she can move to  the position $(2, 1,0) \in \mathcal{P}$. If she starts with $(1,1,0)$, then she wins by eating this last piece.\\
$\mathrm{(ii)}$ By the method that is similar to the one used in $\mathrm{(i)}$, we prove $(ii)$.\\
$\mathrm{(iii)}$ Suppose that $R$ starts with a position in $\mathcal{L}$.\\
$\mathrm{(iii.1)}$ If he starts with a position $(2n+5, 1,0) \in \mathcal{L}_1$ with  $n \in \mathbb{Z}_{\geq0}$, 
there are two cases.\\
$\mathrm{(iii.1.1)}$ If $n=2m$ for $m \in \mathbb{Z}_{\geq0}$, then $2n+5=4(m+1)+1$.
By $(ii)$ of Lemma \ref{lemmaforheight1}, he moves to one of
\begin{equation}
(2m+3,1,0), \cdots, (4m+4,1,0).\label{moveto1} 
\end{equation}
$\mathrm{(iii.1.2)}$ If $n=2m+1$ for $m \in \mathbb{Z}_{\geq0}$, then $2n+5=2(2m+1)+5 = 4(m+1)+3$.
By $(iii)$ of Lemma \ref{lemmaforheight1}, he moves to one of 
\begin{equation}
(2m+3,1,0), \cdots, (4m+6,1,0).\label{moveto2} 
\end{equation}
For any position in  (\ref{moveto1}) and (\ref{moveto2}), the 1st coordinate is larger than or equal to $3$.
Therefore, by Definition \ref{defofL} and Lemma \ref{belongtoN},
any chocolate bar in (\ref{moveto1}) and (\ref{moveto2}) belongs to $\mathcal{L}$ when the 1st coordinate is odd and bigger than 5, and it belongs to $\mathcal{N}$ when the 1st coordinate is three or  an even number bigger than 4.\\
$\mathrm{(iii.2)}$ If he starts with a position in $ \mathcal{L}_2$, we can prove by a method that is similar to the one used in $\mathrm{(iii.1)}$.\\
$\mathrm{(iii.3)}$ If he starts with a position $(1,1,0)$, he loses in the game.\\
$\mathrm{(iv)}$ By the method that is similar to the one used in $\mathrm{(iii)}$, we prove $(iv)$.
\end{proof}

\begin{lemma}\label{frompton}
If we start the game with a chocolate bar $(x,y,s) \in \mathcal{P}$ with $x,y \in \mathbb{N}$ and $s=0,1$, then any move leads to a chocolate bar  $(v,w,t) \notin  \mathcal{P}$, where $v,w \in \mathbb{N}$ and $t=0,1$.
\end{lemma}

\begin{proof}
$\mathrm{(i)}$ Suppose that we start with the position 
\begin{equation}
A=(2^n(2p+5)-1, 2^m(2p+5)-1,s) \in \mathcal{P}_{a}\label{axyp}
\end{equation}
for some $n,m, p \in \mathbb{Z}_{\geq0}$ and $ s =0,1$.
It is sufficient to consider the case when we reduce the first coordinate. 
We prove this by contradiction. We assume that we cut $A$ into $A_1$ and $A_2$, and eat $A_2$ to get $A_1$. \\
$\mathrm{(i.1)}$ Suppose that
$A_1= (2^k(2p+5)-1, 2^m(2p+5)-1,s) \in \mathcal{P}_{a}$ for $k \in \mathbb{Z}_{\geq0}$ such that $k < n$.
Then the first coordinate of $A_2$ is 
\begin{align}
2^n(2p& + 5)-1-(2^k(2p+5)-1) \nonumber \\
&  = (2^n-2^k)(2p+5)  \nonumber \\
&  \geq 2^{n-1}(2p+5)  \nonumber \\
&  > 2^{k}(2p+5)-1. \label{biggerthanhalf}
\end{align}	
By (\ref{biggerthanhalf}) and Lemma \ref{morethanhalf}, the number of $0$ and $1$ in $A_2$ are larger than those in $A_1$.
This contradicts the definition of the game.
Therefore, 
$A_1 \notin  \mathcal{P}_{a}$.\\
$\mathrm{(i.2)}$ Suppose that 
$A_1 \in \mathcal{P}_{b}\cup \mathcal{P}_{c}$. Then,
the second coordinate of $A_1$ is 
\begin{equation}2^m(2p+5)-1=2^k4-1 \label{impossiblemk4}
\end{equation}
or
\begin{equation}
2^m(2p+5)-1=2^k3-1 \label{impossiblemk3}
\end{equation}
for some $k \in \mathbb{N}$, but (\ref{impossiblemk3}) and (\ref{impossiblemk4}) are impossible, 
since $2p+5$ is an odd number larger than or equal to 5. Therefore,
any move does not lead to a position in 
$ \mathcal{P}_{b} \cup \mathcal{P}_{c}$. 

By (\ref{axyp}), the height and the width of $A$ is larger than or equal to $4$, 
any move does not lead to $(1,2,s)$ or $(2,1,s)$ for $s =0,1$.\\
$\mathrm{(ii)}$. Next, we start with the position 
\begin{equation}
A=(x,y)=(2^n3-1,2^m4-1,s) \in \mathcal{P}_{b}\label{axy34}
\end{equation}
 for $n,m \in \mathbb{Z}_{\geq0} \text{ and } s =0,1$.
We cut $A$ into $A_1$ and $A_2$, and eat $A_2$ to get $A_1$. \\

By (\ref{axy34}), we have $x \geq 2$ and $y \geq 3$, and hence 
$A_1 \ne (1,2,s)$ nor $(2,1,s)$ for $s =0,1$. Therefore, we only have to prove that 
 $A_1 \notin \in \mathcal{P}_{a} \cup \mathcal{P}_{b} \cup \mathcal{P}_{c}$.\\
$\mathrm{(ii,1)}$.
First, we reduce the first coordinate.\\
$\mathrm{(ii,1.1)}$.Suppose that 
 $A_1=(2^k3-1, 2^m4-1,s) \in \mathcal{P}_{b}$ for  $k \in \mathbb{Z}_{\geq0}$ such that $k < n$. Then the first coordinate of $A_2$ is 
\begin{align}
 2^n3-&1-(2^k3-1) \nonumber \\
&  \geq 2^{n-1}3   \nonumber \\
&  > 2^k3-1. \label{biggerthanhalf2}
\end{align}	
By (\ref{biggerthanhalf2}) and Lemma \ref{morethanhalf}, the number of $0$ and $1$ in $A_2$ are larger than those in $A_1$.
This contradicts the definition of the game, and hence
 $A_1 \notin \mathcal{P}_{b}$. \\
$\mathrm{(ii,1.2)}$.
If $A_1 \in \mathcal{P}_{a}$, the second coordinate of $A_1$ is
\begin{equation}
2^m4-1=2^k(2p+5)-1 \nonumber
\end{equation}
for $k,p \in \mathbb{Z}_{\geq0}$. This is impossible, and we have $A_1 \notin \mathcal{P}_{a}$.\\
$\mathrm{(ii,1.3)}$.
If $A_1 \in \mathcal{P}_{c}$, the second coordinate of $A_1$ is
\begin{equation}
2^m4-1=2^k3-1 \label{impossiblem43}
\end{equation}
for some  $k \in \mathbb{Z}_{\geq0}$, and this is impossible, and $A_1 \notin \mathcal{P}_{c}$

Next, we reduce the second coordinate.\\
$\mathrm{(ii,2.1)}$ Suppose that 
 $A_1=(2^m3-1, 2^k4-1,s) \in \mathcal{P}_{b}$ for $k < n$ such that $k \in \mathbb{Z}_{\geq0}$, then the first coordinate of $A_2$ is 
\begin{align}
 2^m4-&1-(2^k4-1) \nonumber \\
&  \geq 2^{m-1}4   \nonumber \\
&  > 2^k4-1. \nonumber
\end{align}	
By (\ref{biggerthanhalf2}) and Lemma \ref{morethanhalf}, the number of $0$ and $1$ in $A_2$ are larger than those in $A_1$.
This contradicts the definition of the game. Therefore,
 $A_1=(2^m3-1, 2^k4-1,s) \notin \mathcal{P}_{b}$.\\
$\mathrm{(ii,2.2)}$
If $A_1 \in \mathcal{P}_{a}$, we have 
\begin{equation}
2^m3-1=2^k(2p+5)\nonumber
\end{equation}
for $k,p \in \mathbb{Z}_{\geq0}$. This is impossible, and we have $A_1 \notin \mathcal{P}_{a}$.\\
$\mathrm{(ii,2.3)}$
If $A_1 \in \mathcal{P}_{c}$, we have 
\begin{equation}
2^m3-1=2^k4-1 \nonumber
\end{equation}
for some  $k \in \mathbb{Z}_{\geq0}$, and this is impossible.
Therefore, we have $A_1 \notin \mathcal{P}_{c}$.\\
$\mathrm{(iii)}$
If  we start with the position $(2^n4-1,2^m3-1,s) \in \mathcal{P}_{c}$  for $n,m \in \mathbb{Z}_{\geq0} \text{ and } s =0,1$. Then, we can use a method very similar to the one used in $\mathrm{(ii)}$.\\
$\mathrm{(iv)}$
If  we start with the position $(2,1,s)$ or $(1,2,s) \in \mathcal{P}_{c}$ for $s = 0,1$, then
$L$ moves to $(1,1,1) \in \mathcal{R}$ and $R$ moves to $(1,1,0) \in \mathcal{L}$.
\end{proof}

\begin{lemma}\label{lemmaforP}
We have the following $(i)$ and $(ii)$:\\
$(i)$  If $L$ starts with a position in $\mathcal{P}$, then she always moves to a position in $\mathcal{R} \cup \mathcal{N}$.\\
$(ii)$ If $R$ starts with a position in $\mathcal{P}$, then he always moves to a position in $\mathcal{L}\cup \mathcal{N}$.
\end{lemma}
\begin{proof}
$\mathrm{(i)}$ If $L$ starts with a position in $\mathcal{P}$, by Lemma \ref{frompton}  she cannot move to  a position in $\mathcal{P}$.
Therefore, she always move to a position in $\mathcal{L} \cup \mathcal{R} \cup \mathcal{N}$.

It is sufficient to show that she cannot move to  a position in $\mathcal{L}$. \\
$\mathrm{(i.1)}$ If $L$ starts with $(1,2,s)$ or $(2,1,s)$ with $s = 1,2$, then she always move to $(1,1,1) \in \mathcal{R}$.\\
$\mathrm{(i.2)}$ Suppose that $L$ starts with $A \in \mathcal{P}_{a} \cup \mathcal{P}_{b} \cup \mathcal{P}_{c}$, and moves to the position $B \in \mathcal{L}$. \\
$\mathrm{(i.2.1)}$ 
For $B=(2n+5,1,0)$ or $B=(1,2n+5,0)\in \mathcal{L}$, we prove only the case that 
$B=(2n+5,1,0)\in \mathcal{L}$. By Lemma \ref{morethanhalf} we have
\begin{equation}
 A=(2n+5,2,s) \label{aisby2}
\end{equation}
for $s=1,2$
or
\begin{equation}
A=(m,1,s) \label{aism1}
\end{equation}
for $m \in \mathbb{N}$ such that $m >2n+5$ and $p=0,1$.

When we have (\ref{aisby2}), there is $B^{\prime}=(2n+5,1,1)$ such that $A=B \cup B^{\prime}$. 
Since the number of $1$ in $B^{\prime}$ is larger than that of $B$, by the rule of the game, $L$ cannot eat $B^{\prime}$ to get $B$.
this contradicts the rule of the game.

When we have (\ref{aism1}), by Definition \ref{partizanpposition} we have $A \notin \mathcal{P}$. This leads to a contraction.\\
$\mathrm{(i.2.2)}$ If $B=(1,1,0)$, then $A=(2,1,1) = B \cup (1,1,1)$ or $A=(1,2,1) = B \cup (1,1,1)$. Here, we have a contradiction since $L$ cannot eat $(1,1,1)$.\\
$\mathrm{(ii)}$ We prove $(ii)$ by a method that is similar to the one used in $\mathrm{(i)}$.
\end{proof}

\begin{lemma}\label{lemmaforN}
We have the following $(i)$ and $(ii)$:\\
$(i)$ If $L$ starts the game with a position $(x,y,s) \in \mathcal{N}$, then she can move to a position $(v,w,t) \in  \mathcal{P}\cup \mathcal{L}$.\\
$(ii)$ If $R$ starts the game with a position $(x,y,s) \in \mathcal{N}$, then he can move to a position $(v,w,t) \in  \mathcal{P}\cup \mathcal{R}$.
\end{lemma}

\begin{proof}
$\mathrm{(i)}$
Suppose that $L$ start with the position $A=(x,y,s) \in \mathcal{N}$.
Since $x \in \mathbb{N}$, by Lemma \ref{allnatural}, we have the following cases $\mathrm{(i.1)}$, $\mathrm{(i.2)}$, $\mathrm{(i.3)}$ and $\mathrm{(i.4)}$.\\
$\mathrm{(i.1)}$
Suppose that $x=1$. Then, we start with the position $A=(1,y,s) \in \mathcal{N}$.

By Definition \ref{partizanpposition}, 
Definition \ref{defofL} and Definition \ref{deofR}, we have
$A \ne (1,2,s)$, $A \ne (1,1,s)$ and $A \ne (1,2n+5,s)$ for any $n \in \mathbb{Z}_{\geq0}$ and $s=0,1$. Therefore, we have $A = (1,3,s)$, $(1,4,s)$ or $(1,2n+6,s)$ for $n \in \mathbb{Z}_{\geq0}$ and $s=0,1$.

If $(1,y,s) = (1,3,s)$ or $(1,4,s)$, then $L$ can move to $(1,2,s)$.

If $(1,y,s) = (1,2n+6,s)$, $L$ can move to $(1,2n+5,0)\in \mathcal{L}$.\\
$\mathrm{(i.2)}$
Suppose that $x = 2^n 4-1$ for some $n \in \mathbb{Z}_{\geq0}$.
Then, we have the following two cases $\mathrm{(i.2.1)}$ and $\mathrm{(i.2.2)}$.\\
$\mathrm{(i.2.1)}$
Suppose that there exists $k \in \mathbb{Z}_{\geq0}$ such that
\begin{equation}
2^k3-1 < y < 2^{k+1}3-1.\nonumber
\end{equation}
Since
\begin{equation}
 y \leq 2^{k+1}3-2,\nonumber
\end{equation}
we have 
\begin{equation}
0 < y-(2^k3-1) \leq  2^{k}3-1.\nonumber
\end{equation}
Let $A_1=(x,2^{k}3-1,s)$ and 
$A_2=(x,y-(2^k3-1),s)$.
If $A_1$ and $A_2$ are the same size,
$A_1, A_2 \in \mathcal{P}$. Then,
$L$ can eat the chocolate bar that contains less or equal number of $1$ than or to the other, and get a chocolate bar in $\mathcal{P}$. If $A_1$ is bigger than $A_2$, $L$ eats $A_2$ to get $A_1 \in \mathcal{P}$.\\
$\mathrm{(i.2.2)}$
Suppose that 
\begin{equation}
 y < 2^03-1 = 2.\nonumber
\end{equation}
Then, we have $y=1$, and we prove by a method that is similar to the one used in $\mathrm{(i.1)}$.\\
$\mathrm{(i.3)}$
Suppose that $x = 2^n 3-1$ for some $n \in \mathbb{Z}_{\geq0}$.\\
$\mathrm{(i.3.1)}$
Suppose that there exists $k \in \mathbb{Z}_{\geq0}$ such that
\begin{equation}
2^k4-1 < y < 2^{k+1}4-1.\nonumber
\end{equation}
Then we have 
\begin{equation}
0 < y-(2^k4-1) \leq  2^{k}4-1,\nonumber
\end{equation}
and the rest of the proof of this case is similar to the one in $\mathrm{(i.2.1)}$.\\
$\mathrm{(i.3.2)}$
Suppose that 
\begin{equation}
 y < 2^04-1 = 3.\nonumber
\end{equation}
If $y=1$, we use a method that is similar to the one used in $\mathrm{(i.1)}$.\\
Suppose that $ y= 2 = 2^03-1$.\\
$\mathrm{(i.3.2.1)}$
Suppose that $n \geq 1.$
Since
\begin{equation}
2^n3-1 < 2^n4-2,\nonumber
\end{equation}
we have
\begin{equation}
0 < 2^n3-1-(2^{n-1}4-1) < 2^{n-1}4-1.\nonumber
\end{equation}
Let $A_1=(2^{n-1}4-1,2,s)$ and $A_2=(2^n3-1-(2^{n-1}4-1),2,t)$ with $t=0,1.$
$L$ can eat $A_2$ to get $A_1 \in \mathcal{P}_c$.\\
$\mathrm{(i.3.2.2)}$
Suppose that $n=0.$ Then,
$A=(2,2,s)$, and we let $A_1=(2,1,0)$ and $A_2=(2,1,1)$. Then, we have $A=A_1 \cup A_2$, and 
$L$ can eat $A_1$ to get $A_2$.\\
$\mathrm{(i.4)}$. Suppose that $x = 2^n(2p+5)-1$ for some $n,p \in \mathbb{Z}_{\geq0}$.
Since $(x,y,s) \in \mathcal{N}$, by Definition \ref{partizanpposition}, we have the following 
two cases $\mathrm{(i.4.1)}$ and $\mathrm{(i.4.2)}$.\\
$\mathrm{(i.4.1)}$ Suppose that  there exists $k \in \mathbb{Z}_{\geq0}$
such that 
\begin{equation}
2^k(2p+5)-1 < y < 2^{k+1}(2p+5)-1.\nonumber
\end{equation}
Since $y \leq 2^{k+1}(2p+5)-2$, we have
\begin{equation}
0 < y-(2^k(2p+5)-1) \leq 2^k(2p+5)-1.  \nonumber
\end{equation}
Let $A_1=(x,2^k(2p+5)-1,s)$ and $A_2=(x,y-(2^k(2p+5)-1),t)$ with $t=0$ or $1$. 
If we have 
\begin{equation}
0 < y-(2^k(2p+5)-1) <  2^k(2p+5)-1,  \nonumber
\end{equation}
then by
Lemma \ref{morethanhalf} and Remark \ref{remarkbigsmall},
the number of 1 in $A_2$ is smaller than that of 1 in $A_1$. Therefore,
$L$ can eat $A_2$ to get $A_1 \in \mathcal{P}$.

If we have 
\begin{equation}
y-(2^k(2p+5)-1) =  2^k(2p+5)-1,  \nonumber
\end{equation}
then $A_1$ and $A_2$ are of the same size, and 
$A_1, A_2 \in \mathcal{P}$.
Player $L$ can eat $A_1$ or $A_2$ that contains less or equal $1$ than or to the other to get a chocolate bar in $\mathcal{P}$.\\
$\mathrm{(i.4.2)}$ Suppose that  
\begin{equation}
y < (2p+5)-1=2p+4.\label{smallerthan2p4}
\end{equation}
Since $y \in \mathbb{N}$, by Lemma \ref{allnatural}, 
we have the following cases $\mathrm{(i.4.2.1)}$, $\mathrm{(i.4.2.2)}$, $\mathrm{(i.4.2.3)}$ and $\mathrm{(i.4.2.4)}$.\\
$\mathrm{(i.4.2.1)}$
Suppose that there exist $k,q  \in \mathbb{Z}_{\geq0}$ such that 
\begin{equation}
 y = 2^k(2q+5)-1. \label{yequalkq}   
\end{equation}
By (\ref{smallerthan2p4}), there exists $u \in \mathbb{Z}_{\geq0}$ such that
\begin{equation}
\frac{2p+5}{2^{u+1}} \leq 2^k(2q+5) < \frac{2p+5}{2^u}.\nonumber
\end{equation}
Then, we have 
\begin{equation}
\frac{2p+5}{2} \leq 2^{k+u}(2q+5) < 2p+5,\nonumber
\end{equation}
and hence we have
\begin{equation}
p+3 \leq 2^{k+u}(2q+5) < 2p+5. \label{n2bigp3}
\end{equation}
By (\ref{n2bigp3}) we have 

\begin{align}
2^{k+u}(2q+5) & > p+2 \nonumber \\
& = 2p+5-(p+3) \nonumber \\
& \geq 2p+5-2^{k+u}(2q+5).\label{n2smalp2}
\end{align}
We fix $y$, and reduce the first coordinate $x=2^n(2p+5)-1.$
Let $A_1=(2^n(2^{k+u}(2q+5))-1,y,s)$ and $A_2=(2^n(2p+5 -2^{k+u}(2q+5)),y,t)$ with $t=0,1$.
By (\ref{yequalkq}),$A_1 \in \mathcal{P}$.
By (\ref{n2smalp2}) we have 
\begin{equation}
2^n(2^{k+u}(2q+5))-1 \geq 2^n(2p+5 -2^{k+u}(2q+5)).\nonumber
\end{equation}
If
\begin{equation}
2^n(2^{k+u}(2q+5))-1 >  2^n(2p+5 -2^{k+u}(2q+5)),\nonumber
\end{equation}
by Lemma \ref{morethanhalf} and Remark \ref{remarkbigsmall},
the number of 1 in $A_2$ is smaller than that of 1 in $A_1$. Since $A=A_1 \cup A_2$,
$L$ can get $A_1 \in \mathcal{P}$ by eating $A_2$.
If
\begin{equation}
2^n(2^{k+u}(2q+5))-1 =  2^n(2p+5 -2^{k+u}(2q+5)),\nonumber
\end{equation}
then we also have $A_2 \in \mathcal{P}$, and 
$L$ can choose $A_1$ or $A_2$ according to the number of $1$, and 
eat the one with a smaller number of $1$. If the number of $1$ is the same for both, $L$ can eat either of them.
In this way, she can get a chocolate bar in $\mathcal{P}$.\\
$\mathrm{(i.4.2.2)}$
Suppose that there exist $k \in \mathbb{Z}_{\geq0}$ such that 
$y = 2^k4-1$. Then, we prove by a method that is similar to the one used for $x=2^n4-1$
in $\mathrm{(i.2)}$.\\
$\mathrm{(i.4.2.3)}$
Suppose that there exist $k \in \mathbb{Z}_{\geq0}$ such that 
$y = 2^k3-1$. Then, we prove by a method that is similar to the one used for $x=2^n3-1$ in $\mathrm{(i.3)}$.\\
$\mathrm{(i.4.2.4)}$
Suppose that $y=1$. Then, we prove by a method that is similar to the one used for $x=1$ in $\mathrm{(i.1)}$.\\
$\mathrm{(ii)}$ We can prove the case for the case of Player $R$ by a method that is very similar to the one used for Player $L$ in $\mathrm{(i)}$.
\end{proof}

\begin{theorem}
We have the following $(i)$, $(ii)$, $(iii)$, and $(iv)$:\\
$(i)$ The set $\mathcal{P}$ is the set of $\mathcal{P}$-positions.\\
$(ii)$ The set $\mathcal{N}$ is the set of $\mathcal{N}$-positions.\\
$(iii)$ The set $\mathcal{L}$ is the set of $\mathcal{L}$-positions.\\
$(iv)$ The set $\mathcal{R}$ is the set of $\mathcal{R}$-positions.\\
\end{theorem}
\begin{proof}
By Lemma \ref{howtodecide} and  Lemma \ref{lemmaforP}, we have $(i)$,
and by  Lemma \ref{howtodecide} and Lemma \ref{lemmaforN}, we have $(ii)$.

By Lemma \ref{howtodecide} and  Lemma \ref{decideLorR}, we have $(iii)$ and $(iv)$.
\end{proof}

\begin{figure}[H]
\begin{minipage}[t]{0.5\textwidth}
\begin{center}
\includegraphics[height=3.5cm]{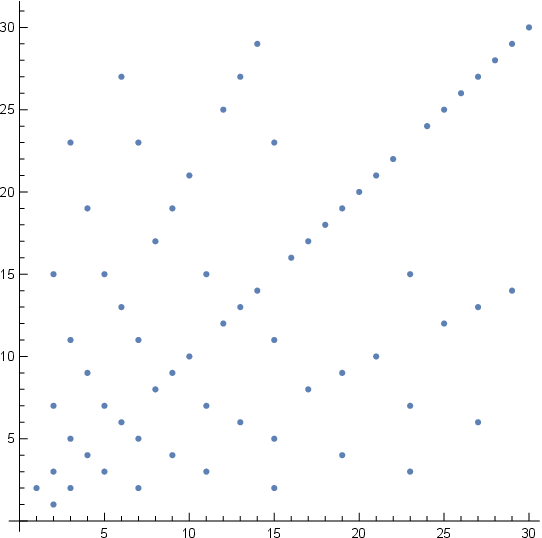}
\caption{P-positions of \\
a partizan chocolate bar  game.}
\label{parp}
\end{center}
\end{minipage}
\begin{minipage}[t]{0.5\textwidth}
\begin{center}
\includegraphics[height=3.5cm]{maxchocop.eps}
\caption{P-positions of \\
an impartial chocolate bar  game}
\label{maxp2}
\end{center}
\end{minipage}
\end{figure}
The set of  $\mathcal{P}$-positions of a  partizan chocolate bar game is different from, but is very similar to the set of 
$\mathcal{P}$-positions of an impartial chocolate bar game.
See Figure \ref{parp} and Figure \ref{maxp2}.

\section{The Prospect of Further Research}
The authors are studying three-dimensional chocolate games, and the impartial version is an easy generalization, but the partizan version is a challenging problem.
The authors have not managed to find formulas for the set of previous players' positions.

\section*{Acknowledgements}
This work was supported by JSPS KAKENHI Grant Number 23H05173.

We would like to thank Prof. R. J. Nowakowski for his useful advice, and we 
we would like to thank Editage (www.editage.com) for English language editing.

\end{document}